\documentclass [twoside,reqno, 12pt] {amsart}

\usepackage{amsmath, amssymb, amsthm, hyperref, graphicx, setspace, mathrsfs, amsfonts}
\usepackage[a4paper, margin=1in]{geometry}  

\newtheorem{theorem}{Theorem}[section]

\newtheorem{proposition}[theorem]{Proposition}

\theoremstyle{definition}

\newtheorem{remark}[theorem]{Remark}

\title[]{\small \textbf{Anisotropic Calder\'{o}n Problem for a Non-Local Second Order Elliptic Operator}}

\author[ ]{\small  Susovan Pramanik }
\date{}

\address{Susovan Pramanik, Harish-Chandra Research Institute, A CI of Homi Bhabha National Institute, Chhatnag Road, Jhunsi, Allahabad 211 019, India}
\email{susovanpramanik@hri.res.in}

\begin{document}

	\begin{abstract}
		
		\noindent
		This paper investigates the anisotropic Calder\'{o}n problem for a non-local elliptic operator of order 2, on closed Riemannian manifolds. 
		We demonstrate that using the Cauchy data set, we can recover the geometry of a closed Riemannian manifold up to standard gauge.  
	\end{abstract}
	
	\maketitle

	\section{Introduction}

	The anisotropic Calder\'{o}n problem is geometric in nature (cf. \cite{LU89, Sal13}). Let us state the well-known open problem in the context of a closed, connected Riemannian manifold $(M, g)$ with $\dim(M) > 2$. Let $\mathcal{O} \subset M$ be a non-empty open set. The Calder\'{o}n problem examines whether knowledge of the Cauchy data set 
	\begin{align*}
		\mathcal{C}^\mathcal{O}_{M, g}= \{(u|_\mathcal{O}, (-\Delta_g)u|_\mathcal{O}) \mid u \in C^\infty(M), \, -\Delta_g u = 0 \,\text{in} \,\, M\setminus \overline{\mathcal{O}}\}
	\end{align*}
	determines the isometry class of the manifold $(M, g)$. \((-\Delta_g) \) represents the positive Laplace-Beltrami operator on \( (M, g)\). 
	
	\medskip
	When $\dim(M)= 2$, the problem is solved using an additional gauge due to conformal invariance of the Laplace-Beltrami operator. See some of the references \cite{Nach96, KL06, Buk08, GT11}. The problem for real-analytic manifolds in dimensions three and above has been addressed (see the references \cite{LU89,  LU01, LTU03}), but it is still open for smooth manifolds. Positive results have been observed in transversally anisotropic geometries, see \cite{DKSU09, DKLS16}. Starting from these celebrated works \cite{SU87, KSU07}, for current research on the Calder\'{o}n problem, we refer to the survey article \cite{Uhl14}.
	
	\medskip
	The purpose of this work is to explore the anisotropic Calder\'{o}n problem for a second order elliptic psedodifferential operator $((-\Delta_g)^2+m^2\mathbb{I})^\frac{1}{2}$ for some fixed $m\in \mathbb{R}\setminus \{0\}$, on smooth, closed, connected Riemannian manifolds.
	
	\medskip
	Let $(M, g)$ be a closed, connected Riemannian manifold with dimension \(\dim(M) \geq 2 \). We define \(-\Delta_g \) as the positive definite, self-adjoint Laplace-Beltrami operator on \(L^2(M) \) with domain $\mathcal{D}(-\Delta_g) = H^2(M);$ \cite[pp. 25]{Tay23}.
	
	The spectrum of \( (-\Delta_g) \) over $(M, g)$ consists of a discrete sequence of eigenvalues  
	\[
	0 = \lambda_0 < \lambda_1 \leq \lambda_2 \leq \cdots \to+\infty,
	\]  
	where each eigenvalue \( \lambda_k \) has a finite multiplicity \( d_k \).
	
	\medskip
	Define \(\pi_k: L^2(M) \to \ker(-\Delta_g - \lambda_k) \) as the orthogonal projection onto the eigenspace corresponding to \(\lambda_k \). Then, for any function \(f \in L^2(M) \), 
	\[
	\pi_k(f) = \sum_{j=1}^{d_k} \langle f, \phi_{k_j} \rangle _{L^2(M)} \phi_{k_j},
	\]  
	where, \( \{\phi_{k_j} \}_{j=1}^{d_k} \) forms an orthonormal basis for \( \ker(-\Delta_g - \lambda_k) \).  
	
	\medskip
	Set $m\in \mathbb{R}\setminus\{0\}$. We consider the positively perturbed bi-Laplace operator on $(M,g)$ as:
	\[ \mathcal{L}_g:= (-\Delta_g)^2 + m^2\mathbb{I}. \]  
	This is an unbounded operator on \(L^2(M) \), with the domain of definition $\mathcal{D}(\mathcal{L}_g) = H^4(M)$, see \cite{Tay81}.  
	
	\medskip
	We now define the operator of our interest, the $\frac{1}{2}$-fractional power of \(\mathcal{L}_g \), through its action over its domain of definition $\mathcal{D}(\mathcal{L}_g ^{\frac{1}{2}})=H^2(M)$ (cf. \cite{Tay81}), as
	\begin{equation}\begin{aligned}
			\mathcal{L}_g^{\frac{1}{2}} u &= \left( (-\Delta_g)^2 + m^2 \right)^{\frac{1}{2}} u,\quad u\in H^2(M) \\
			&= \sum_{k=0}^\infty (\lambda_k^2 + m^2)^{\frac{1}{2}} \pi_k (u).
	\end{aligned}\end{equation}
	We note that the operator $\mathcal{L}_g ^{\frac{1}{2}}$ is an elliptic pseudodifferential with order $2$. For related discussion, see Section \ref{sym-ord}. 
	
	\medskip
	Because all eigenvalues of \(\mathcal{L}_g \) are strictly positive, the operator \(\mathcal{L}_g ^{\frac{1}{2}} \) is invertible over the range space $\mathcal{R}(\mathcal{L}_g)=L^2(M)$, and its inverse is given by  
	\begin{equation}\begin{aligned}\label{L-inv}
			\mathcal{L}_g^{-\frac{1}{2}} f &= \left( (-\Delta_g)^2 + m^2 \right)^{-\frac{1}{2}} f,\quad f\in L^2(M) \\
			&= \sum_{k=0}^\infty (\lambda_k^2 + m^2)^{-\frac{1}{2}} \pi_k (f).
	\end{aligned}\end{equation}
	Let \( \mathcal{O} \subset M \) be a non-empty open subset of \( M \).  For each \( f \in C_0^\infty(\mathcal{O}) \), the equation  
	\begin{equation}
		\mathcal{L}^{\frac{1}{2}}_g u = f \quad \mbox{in }  M \label{eq:1.1}
	\end{equation}  
	has a unique solution \( u = u_f \in C^\infty(M) \cap L^2(M)\). See \cite{FGKU25}. The solution is formally expressed as $u= \mathcal{L}_g ^{-\frac{1}{2}}f$ (cf. \eqref{L-inv}). 
	
	\medskip
	Now we state our main result.
	\begin{theorem}\label{th1.1}
		Let $(M_j, g_j)$ be two smooth, closed, connected Riemannian manifolds, with $M_1 \cap M_2 \neq \emptyset$. Consider $\mathcal{O}_j \subset M_j, \, j = 1, 2$ to be two non-empty open sets such that $(\mathcal{O}_1, g_1) = (\mathcal{O}_2, g_2) =: (\mathcal{O}, g)$. Assume that \[ \mathcal{L}_{g_1} ^{-\frac{1}{2}} f|_\mathcal{O} = \mathcal{L}_{g_2} ^{-\frac{1}{2}} f|_\mathcal{O}, \quad \forall f \in C^\infty_0(\mathcal{O} ). \]
		Then, there exists a diffeomorphism $\Phi: M_1 \to M_2$ such that $\Phi^\ast g_2 = g_1$.
	\end{theorem}
	\begin{remark}
		In recent years, the research of non-local inverse problems has been quite dynamic. Starting from the article \cite{GSU20}, without being exhaustive we mention some
		references \cite{GLX17, RS20, RS2020, GRSU20, MLR20, Cov20, Li20, HL20, BGU21, KLW22, CGR22, Gho22, Zim23, HU24, HLW24, Das25}. However, the operator's order (cf. Section \ref{sym-ord}) in consideration falls within $(0, 2)$. This study extends that limit to $2$ using a straightforward example of the square root of a positively perturbed bi-Laplace operator. 
		
		The fractional power of the perturbed Laplace operator is known as the relativistic Schr\"{o}dinger operator $((-\Delta_g) + m^2\mathbb{I})^\alpha$, $\alpha\in (0,1)$, and has its due significance in mathematical physics. We refer to the work of V. Ambrosio \cite{Amb22, Amb23} on this account. 
	\end{remark}
	\begin{remark}
		Let $m\in (0, \infty)$, and we introduce another operator of interest, 
		\[ \mathcal{A}_g:= \mathcal{L}_g ^{\frac{1}{2}} - m\mathbb{I}.\]
		The definition is: 
		\[ \mathcal{A}_g u = \sum_{k=0}^\infty \left((\lambda_k^2 + m^2)^{\frac{1}{2}} - m \right) \pi_k u, \]  
		where \(u \in \mathcal{D}(\mathcal{A}_g)=H^2(M) \), and $\{(\lambda_k, \pi_k)\}_{k=0}^\infty$ as previously. 
		
		\medskip
		Let \(u\in C^\infty(M)\cap L^2(M) \) be the unique solution of the equation $$\mathcal{A}_gu=0\mbox{ in }M$$ with satisfying $\langle u, 1\rangle_{L^2(M)}=0$ for \( f \in C_0^\infty(\mathcal{O}) \). See \cite{FKU24}.
		
		\medskip
		Next, we define the corresponding Cauchy data set: 
		\[ \mathcal{C}^{\mathcal{O}}_{M, g}= \{(u|_\mathcal{O},  \mathcal{A}_g u|_\mathcal{O}) \mid u \in C^\infty(M), \,\,\, \mathcal{A}_g u = 0 \,\text{in} \, M \setminus \overline{\mathcal{O}} \}. \]  
		This dataset contains solutions to the equation \(\mathcal{A}_g u = 0 \) and serves as the basis on the anisotropic Calder\'{o}n problem for the operator $\mathcal{A}_g$.
		\begin{theorem}
			Let \( (M_i, g_i), \, i = 1,2 \) be two closed, connected Riemannian manifolds with \( \dim(M_i) \geq 2 \). Assume that \(\mathcal{O} \subset M_1 \cap M_2 \) is a non-empty open subset where the metrics agree, i.e., $(\mathcal{O}, g_1) = (\mathcal{O}, g_2):= (\mathcal{O}, g)$. The equality of the Cauchy data set over $(\mathcal{O}, g)$, i.e. $\mathcal{C}^\mathcal{O}_{M_1, g_1} = \mathcal{C}^\mathcal{O}_{M_2, g_2}$, implies the existence of a diffeomorphism $\Phi: M_1 \to M_2$ such that $\Phi^\ast g_2 = g_1$.
		\end{theorem}
		
		\noindent
		The proof is quite similar to Theorem \ref{th1.1}. 
	\end{remark}
	
	\medskip
	This paper's results stem from the work \cite{GU21} to solve the nonlocal Calder\'{o}n's problem of recovering anisotropic medium or the metric. Subsequent development adds \cite{CGRU23, Feiz24, CO24, FGKU25, FKU24, FGKRSU25, CR25} to further complete the picture in this nonlocal regime. In particular, we now know that, 
	the knowledge of the Cauchy data set 
	\begin{align*}
		\mathcal{C}^{\alpha, \mathcal{O}}_{M, g}= \{(u|_\mathcal{O}, (-\Delta_g)^\alpha u|_\mathcal{O}) \mid u \in C^\infty(M), \, (-\Delta_g)^\alpha u = 0 \,\text{in} \,\, M\setminus \overline{\mathcal{O}}, \,\, \alpha\in (0,1)\}
	\end{align*}
	determines the manifold $(M, g)$ uniquely up to a diffeomorphism. This applies to both closed and open manifolds, see \cite{FGKU25, FGKRSU25}. This is based on reducing the problem into the hyperbolic inverse problem, and thanks to the work of \cite{HLOS18} for instance, we have a definite conclusion in the nonlocal analogue of the anisotropic Calder\'{o}n problem, which is still to be discovered for the local anisotropic Calder\'{o}n problem.

	In general, these nonlocal problems are motivated by many models, such as diffusion process \cite{AMRT10}, image processing \cite{GO08}, stochastic process \cite{BV16} and so on.

	\medskip
	The paper is organized as follows. In Section \ref{sec2}, we introduce the fractional power of the elliptic operator using the semigroup approach. Further, we emphasize the symbol and order of these non-local operators.  Section \ref{sec3} is devoted to the proof of our main result.
	
	\section{Preliminary understandings.}\label{sec2}
	\subsection{Heat semigroup approach in defining the fractional power of $\mathcal{L}_g$}
	Here, we define the fractional power of $\mathcal{L}_g = ((-\Delta_g)^2 + m^2\mathbb{I})$ using the semigroup technique, which is equivalent to the prior definition.

	\medskip
	Let \(e^{-t\mathcal{L}_g} \) denote the associated heat semigroup on \(L^2(M)\). Let $v\in L^2(M)$, we define \begin{equation}\label{hk1} e^{-t \mathcal{L}_g } v(x) = \int_M \mathcal{K}_{\mathcal{L}_g }(t,x, y) v(y) \, dV_g(y), 
	\end{equation}
	$\mathcal{K}_{\mathcal{L}_g }(t,x,y)$ is the heat kernel for the heat semigroup $e^{-t\mathcal{L}_g }$, which is defined as 
	\begin{equation} \begin{aligned}
			\mathcal{K}_{\mathcal{L}_g }(t,x,y) = \sum_{k=0}^\infty e^{-t(\lambda_k^2 +m^2)} \,\phi_k(x)\,\phi_k(y)&=e^{-m^2t} \sum_{k=0}^\infty e^{-t\lambda_k^2 } \,\phi_k(x)\,\phi_k(y)\\ &=e^{-m^2t} \mathcal{K}_{\Delta_g^2}(t,x,y),\label{2.1} \end{aligned}\end{equation} 
	where, $K_{\Delta^2_g}(t,x,y)$ be the heat kernel for the heat semigroup $e^{-t\Delta^2_g}$.
	
	\medskip
	Fei He \cite{Fei22} provides a pointwise estimate of the heat kernel $\mathcal{K}_{\Delta_g^2}(t,x,y)$.
	\begin{theorem}[\cite{Fei22}]\label{th2.1}
		Let $x, y$ be two points on an arbitrary smooth connected compact Riemannian manifold $M$, and let $t \in (0,\infty)$. Then 
		\begin{equation}
			\left | \mathcal{K}_{\Delta_g^2}(t, x, y) \right | \leq \frac{C}{t^{n/4}} e^{ -\frac{c\,d_g (x, y)^\frac{4}{3}}{t^\frac{1}{3}}  },
		\end{equation}
		where $C>0$, and \( d_g(x, y) \) is the Riemannian distance between \( x \) and \( y \).
	\end{theorem}
	
	\noindent
	According to \cite{Fei22}, $\int_M \mathcal{K}_{\Delta_g^2}(t,x,y) \,\,dV_g =1$, yields  $\int_M \mathcal{K}_{\mathcal{L}_g }(t,x,y) \,\,dV_g \leq 1$, showing that $\{e^{-t \mathcal{L}_g}\}_{t \geq 0}$ is a submarkovian semigroup. Using Theorem \ref{th2.1}, we obtain the following estimate:
	\begin{equation}
		\left | \mathcal{K}_{\mathcal{L}_g }(t,x,y) \right| \leq \frac{C}{t^{n/4}} \,e^{-m^2t}\,e^{ -\frac{c\,d_g (x, y)^\frac{4}{3}}{t^\frac{1}{3}}},\quad \forall (t, x, y)\in (0,\infty)\times M\times M.\label{2.2}
	\end{equation}
	
	\subsection*{Euclidean case in $(\mathbb{R}^n, e)$:}
	The heat kernel for the bi-Laplace heat equation has the self-similar form:
	\[ \mathcal{K}_{\Delta_e^2}(t, x, y)=\frac{1}{t^{\frac{n}{4}}} \, e^{-c_n\frac{|x - y|^{\frac{4}{3}}}{t^{\frac{1}{3}}}}, \quad\forall (t, x, y)\in (0,\infty)\times \mathbb{R}^n \times \mathbb{R}^n,\]
	where $c_n>0$ is some constant.

	\medskip
	So for the operator $(-\Delta_e)^2 + m^2\mathbb{I}$ in \( (\mathbb{R}^n, e) \), we have 
	$$e^{-t((-\Delta_e)^2 + m^2\mathbb{I})} v(x) = e^{-m^2t} \,\frac{1}{t^{\frac{n}{4}}} \int_{\mathbb{R}^n} e^{-c\frac{|x - y|^{\frac{4}{3}}}{t^{\frac{1}{3}}}} v(y) \, dy,  $$
	for $v\in L^2(\mathbb{R}^n)$.
	
	\medskip
	This can be seen that, 
	\begin{align*}
		e^{-t((-\Delta_e)^2 + m^2\mathbb{I})} v(x) 
		&=  e^{-m^2t}\int_{\mathbb{R}^n} \,\frac{1}{t^{\frac{n}{4}}} R\left(\frac{x - y}{t^{\frac{1}{4}}}\right) v(y) \, dy,
		\quad \text{where, } R(x) = e^{-c|x|^{\frac{4}{3}}} \\[3pt]
		\Rightarrow \quad \left| e^{-t ((-\Delta_e)^2 + m^2\mathbb{I})} v(x) \right| &\leq e^{-m^2t} \, \Big|\Big|\frac{1}{t^{\frac{n}{4}}} R\left(\frac{x - y}{t^{\frac{1}{4}}}\right) \Big|\Big|_{L^2(\mathbb{R}_y^n)} \, \|v\|_{L^2(\mathbb{R}^n_y)}\\[3pt]
		&\quad= e^{-m^2t} \, \|R\|_{L^2(\mathbb{R}^n)} \, \|v\|_{L^2(\mathbb{R}^n)}.
	\end{align*}
	
	\medskip
	Next, we define the fractional power of $((-\Delta_e)^2 + m^2\mathbb{I})$ in \( (\mathbb{R}^n, e) \). Consider  \( \alpha \in (0,1) \) and  $v \in L^2(\mathbb{R}^n)$.  We first define
	\begin{equation*}
		((-\Delta_e)^2 + m^2\mathbb{I})^{-\alpha} v = \frac{1}{\Gamma(\alpha)} \int_0^{\infty} t^{\alpha - 1} e^{-t((-\Delta_e)^2 + m^2\mathbb{I})} v \, dt.
	\end{equation*}
	This is pointwise well-defined, since
	\begin{align*}
		\left| ((-\Delta_e)^2 + m^2\mathbb{I})^{-\alpha} v(x) \right| 
		&= \left| \frac{1}{\Gamma\left(\alpha\right)} \int_0^\infty t^{-\frac{1}{2}} e^{-t ((-\Delta_e)^2 + m^2\mathbb{I})} v(x) \, dt \right| \\
		&\leq \frac{1}{\Gamma\left(\alpha\right)}\|R\|_{L^2(\mathbb{R}^n)} \, \|v\|_{L^2(\mathbb{R}^n)}\, \int_0^\infty e^{-m^2t} \frac{dt}{t^{1-\alpha}} \\
		&\quad= \frac{1}{\Gamma\left(\alpha\right)}\|R\|_{L^2(\mathbb{R}^n)} \, \|v\|_{L^2(\mathbb{R}^n)} \, \frac{\Gamma\left(\alpha\right)}{|m|^{2\alpha}}\\
		&\quad= \, \frac{\|R\|_{L^2(\mathbb{R}^n)} \, \|v\|_{L^2(\mathbb{R}^n)}}{|m|^{2\alpha}} 
		< \infty, \quad m\in \mathbb{R}\setminus\{0\}.
	\end{align*}
	
	\noindent
	In other words, we have the mapping property
	\[  ((-\Delta_e)^2 + m^2\mathbb{I})^{-\alpha}: L^2(\mathbb{R}^n)\mapsto L^\infty(\mathbb{R}^n),\quad\mbox{with } \| ((-\Delta_e)^2 + m^2\mathbb{I})^{-\alpha} \|_{L^2\mapsto L^\infty} \lesssim \frac{1}{|m|^{2\alpha}}. \]
	
	\medskip
	Further one defines, for $\alpha\in (0,1)$ and $u \in H^4(\mathbb{R}^n)$,
	\begin{align*}
		((-\Delta_e)^2 + m^2\mathbb{I})^{\alpha} u &= ((-\Delta_e)^2 + m^2\mathbb{I})\circ ((-\Delta_e)^2 + m^2\mathbb{I})^{-(1 - \alpha)} u \\
		&=  ((-\Delta_e)^2 + m^2\mathbb{I})^{-(1 - \alpha)}\circ ((-\Delta_e)^2 + m^2\mathbb{I}) u\\
		&=  \frac{1}{\Gamma(1 - \alpha)} \int_0^{\infty} t^{-\alpha} e^{-t((-\Delta_e)^2 + m^2\mathbb{I})} (((-\Delta_e)^2 + m^2\mathbb{I}) u) \, dt.
	\end{align*}
	However, $H^4(\mathbb{R}^n)\subseteq \mathcal{D}\left(((-\Delta_e)^2 + m^2\mathbb{I})^{\alpha}\right)=H^{4\alpha}(\mathbb{R}^n)$, for $\alpha\in (0,1)$, can be seen easily through the Fourier transform definition. We refer \cite{GSU20} for this approach and the related mapping properties.
	
	\subsection*{In $(M,g)$:} We would like to do the same exercise for the closed manifold $(M, g)$ that we are interested in. From \eqref{hk1} and \eqref{2.2}, we have the pointwise bound on $e^{-t \mathcal{L}_g } v$ as  
	\begin{align}
		|e^{-t \mathcal{L}_g } v(x)| 
		&= \left| \int_M \mathcal{K}_{\mathcal{L}_g }(t,x, y) v(y) \, dV_g(y) \right| \notag\\
		&\lesssim e^{-m^2 t}\, \int_M  \frac{1}{t^{n/4}}R\left( \frac{d(x, y)}{t^{1/4}} \right) |v(y)| \, dV_g(y), \notag\\
		&\lesssim e^{-m^2 t}\, \|R\|_{L^2(M)} \, \|v\|_{L^2(M)}\label{eL}
	\end{align}
	for $v\in L^2(M)$.
	
	\medskip
	Now we define the fractional power of $\mathcal{L}_g $. Let  \( \alpha \in (0,1) \). We define $\mathcal{L}_g ^{-\alpha}$ over $L^2(M)$ as 
	\begin{equation}\label{L-alpha}
		\mathcal{L}_g ^{-\alpha} v = \frac{1}{\Gamma(\alpha)} \int_0^{\infty} t^{\alpha - 1} e^{-t\mathcal{L}_g } v \, dt.
	\end{equation}
	for $v \in L^2(M)$
	
	\medskip
	Using \eqref{eL}, as before it can be seen that
	$$\left| \mathcal{L}_g ^{-\alpha} v(x) \right| \leq C  \, \frac{ \|R\|_{L^2(M)} \, \|u\|_{L^2(M)}}{|m|^{2\alpha}}   < \infty, \quad m \in \mathbb{R} \setminus \{0\},$$
	or, the mapping property follows as
	\[ \mathcal{L}_g ^{-\alpha}: L^2(M)\mapsto L^\infty(M),\quad \| \mathcal{L}_g ^{-\alpha} \|_{L^2\mapsto L^\infty}\lesssim \frac{1}{|m|^{2\alpha}}. \]

	\medskip
	Let us recall \eqref{hk1} here, since all the eigenvalues of $\mathcal{L}_g $ are positive, and $0<(\lambda_1+m^2)\leq(\lambda_2+m^2)\leq\cdots\leq(\lambda_k+m^2)\to \infty$, then $\exists \,\,\beta>0$ such that
	$$\|e^{-t\mathcal{L}_g } \,v(x)\|_{L^2(M)} \leq e^{-\beta t}\,\|v(x)\|_{L^2(M)},$$
	for $v\in L^2(M)$.
	
	\medskip
	
	Now by using Minkowski's integral inequality, we get
	\begin{align}
		\|\mathcal{L}_g ^{-\alpha} v\|_{L^2(M)} 
		&= \left( \int_M \left| \frac{1}{\Gamma({\alpha})} \int_0^\infty t^{-(1-\alpha)} e^{-t\mathcal{L}_g } v(x) \, dt \right|^2 dV_g(x) \right)^{ \frac{1}{2} } \notag\\
		&\leq \frac{1}{\Gamma({\alpha})} \int_0^\infty t^{-(1-\alpha)} \left( \int_M |e^{-t\mathcal{L}_g } v(x)|^2 \, dV_g(x) \right)^{ \frac{1}{2} } dt \notag\\
		&\quad= \frac{1}{\Gamma({\alpha})} \int_0^\infty t^{-(1-\alpha)} \|e^{-t\mathcal{L}_g } v\|_{L^2(M)} \, dt \notag\\
		&\qquad\leq \frac{\|v\|_{L^2(M)}}{\Gamma({\alpha})} \int_0^\infty t^{-(1-\alpha)} e^{-\beta t} \, dt \notag\\
		&\qquad= \frac{\|v\|_{L^2(M)}}{\Gamma(\alpha)} \, \Gamma(\alpha)\,\frac{1}{\beta^{\alpha}} = \frac{\|v\|_{L^2(M)}}{\beta^{\alpha}}.\label{cv}
	\end{align}
	Hence, the mapping property 
	\[ \mathcal{L}_g ^{-\alpha}: L^2(M)\mapsto L^2(M),\quad\mbox{with } \| \mathcal{L}_g ^{-\alpha} \|_{L^2\mapsto L^2}\leq \,\frac{1}{\beta^{\alpha}}. \]

	\medskip
	Next, we define $\mathcal{L}_g ^{\alpha}$ over $H^4(M)$  as
	\[
	\mathcal{L}_g ^{\alpha} u = \mathcal{L}_g  \circ \mathcal{L}_g ^{-(1 - \alpha)} u =  \mathcal{L}_g ^{-(1 - \alpha)}\circ \mathcal{L}_g  u=  \frac{1}{\Gamma(1 - \alpha)} \int_0^{\infty} t^{-\alpha} e^{-t\mathcal{L}_g } (\mathcal{L}_g  u) \, dt, 
	\]
	for $u \in H^4(M)\subseteq \mathcal{D}(\mathcal{L}_g ^{\alpha})$, and  $\alpha\in (0,1)$. See \cite{CO24}.
	
	\subsection{Symbol and order}\label{sym-ord}
	Now, let us discuss the symbol and order of the operator $\mathcal{L}_{g} ^{\frac{1}{2}}$ defined over $(M,g)$. Up to the level of chart, let us summarise our understanding of the operator $\mathcal{L}_{g} ^{\frac{1}{2}}:=((-\Delta_g)^2+m^2\mathbb{I})^{\frac{1}{2}}$ over $(\mathbb{R}^n, g)$. We begin with the following result
	\begin{theorem}[\cite{FGKRSU25}]
		Let  $\beta > 0$ and $g$ be a $C^\infty$ Riemannian metric on $\mathbb{R}^n$, $n \geq 2$, which agrees with the Euclidean metric outside of a compact set. Then the operator $(-\Delta_g)^\beta\in \Psi^{2\beta}_{1,0}(\mathbb{R}^n)$, and is a classical elliptic pseudodifferential operator on $\mathbb{R}^n$, with principal symbol 
		\[
		\left(\sum_{j,k=1}^n g^{jk}(x)\xi_j\xi_k\right)^\beta, \quad (x,\xi) \in T^*\mathbb{R}^n \setminus \{0\}.
		\] 
	\end{theorem}
	\begin{remark}
		In the preceding finding, we understand $(-\Delta_g)^\beta$ as follows. Let $\beta=m+\alpha$, where $m\in\mathbb{Z}$ and $\alpha\in (0,1)$, then $(-\Delta_g)^\beta=(-\Delta_g)^m\circ (-\Delta_g)^\alpha$, where the operator $(-\Delta_g)^\alpha$ is defined via the functional calculus as follows. 
		
		\medskip
		Consider the operator $(-\Delta_g)$ in $(\mathbb{R}^n, g)$. Let \(e^{t\Delta_g} \) denote the associated heat semigroup on \(L^2(\mathbb{R}^n)\). Let $\alpha\in (0,1)$ and $v\in L^2(\mathbb{R}^n)$. We define \begin{equation*}
			(-\Delta_g)^{-\alpha}v(x)=\frac{1}{\Gamma(\alpha)}\int_0^\infty t^{\alpha-1}\,e^{t \Delta_g} v(x)\, dt. \end{equation*}
		The operator $(-\Delta_g)^{\alpha}$ over $H^2(\mathbb{R}^n)$ for $\alpha\in (0,1)$ is defined as \[ (-\Delta_g)^{\alpha}u(x)=\frac{1}{\Gamma(1-\alpha)}\int_0^\infty t^{-\alpha}\,e^{t \Delta_g} ((-\Delta_g)u(x))\,dt.
		\]
		However, $H^2(\mathbb{R}^n)\subseteq \mathcal{D}\left((-\Delta_g)^\alpha\right)=H^{2\alpha}(\mathbb{R}^n)$ for $\alpha\in (0,1)$. We refer to \cite{FGKU25} for further details.
	\end{remark}
	Next, we review the symbol class definitions.
	
	\medskip
	Assume \(m, \rho, \delta \in \mathbb{R} \) and \(0 \leq \rho, \delta \leq 1 \). We define the symbol class \( S^m_{\rho,\delta}(\mathbb{R}^n \times \mathbb{R}^n \times \mathbb{R}^n) \) to consist of all functions \( p \in C^\infty(\mathbb{R}^n \times \mathbb{R}^n \times \mathbb{R}^n) \) such that, for every compact set \( K \subset \mathbb{R}^n \) and for all multi-indices \( \alpha, \beta \), there exists a constant \( C_{K,\alpha,\beta} > 0 \) such that
	\[
	|\partial_x^\alpha \partial_\xi^\beta p(x, \xi)| \leq C_{K,\alpha,\beta} (1 + |\xi|)^{m - \rho |\beta| + \delta |\alpha|}, \quad \text{for all } x \in K,\, \xi \in \mathbb{R}^n.
	\]
	
	\medskip
	In particular, the subclass \( S^m_{1,0}(\mathbb{R}^n) \) consists of symbols \( p(x, \xi) \in C^\infty(\mathbb{R}^n \times \mathbb{R}^n \times \mathbb{R}^n)) \) for which the estimates
	\[
	|\partial_x^\alpha \partial_\xi^\beta p(x, \xi)| \leq C_{K,\alpha,\beta} (1 + |\xi|)^{m - |\beta|}
	\]
	hold uniformly for all \( x \in K  \), where \( K \) is any compact set.
	
	The space \( \Psi^m_{1,0}(\mathbb{R}^n) \) of pseudodifferential operators on \( \mathbb{R}^n \)
	associated with the symbol class \( S^m_{1,0}(\mathbb{R}^n \times \mathbb{R}^n \times \mathbb{R}^n) \)
	consisting of operators of the form
	\[
	A u(x) = \frac{1}{(2\pi)^n} \int_{\mathbb{R}^n} \int_{\mathbb{R}^n} e^{i(x - y)\cdot \theta} \, a(x, y, \theta) \, u(y) \, dy \, d\theta, \quad u \in C^\infty_0(\mathbb{R}^n),
	\]
	where the amplitude \( a(x, y, \theta) \in S^m_{1,0}(\mathbb{R}^n \times \mathbb{R}^n \times \mathbb{R}^n) \). For more details, we refer to  \cite {GS94}.

	\subsection*{Symbol class for $\mathcal{L}_g ^{\frac{1}{2}}$}
	We have 
	\[
	\mathcal{L}_g  = \left(-\sum_{j,k=1}^n g^{jk}(x) \partial_j \partial_k \right)^2 + m^2\mathbb{I},\quad m\in\mathbb{R}\setminus\{0\}.
	\]
	The symbol for \(\mathcal{L}_g \) is 
	\[ p(x, \xi) = \left(\sum_{j,k=1}^n g^{jk}(x) \xi_j \xi_k \right)^2 + m^2.\]
	This is a classical example of a symbol in \(S^2_{1,0}(\mathbb{R}^n \times \mathbb{R}^n) \), since it fulfills the symbol estimates: \[ |\partial_x^\alpha \partial_\xi^\beta p(x, \xi)| \leq C_{\alpha, \beta} (1 + |\xi|)^{2 - |\beta|}, \] for all multi-indices \(\alpha, \beta \). 
	
	Thus, \[ \mathcal{L}_g ^{ \frac{1}{2} } \in \Psi^2_{1,0}(\mathbb{R}^n) \]
	with its symbol
	$$\left(\left(\sum_{j,k=1}^n g^{jk}(x) \xi_j \xi_k \right)^2 + m^2\right)^{\frac{1}{2}}, \quad (x,\xi) \in T^*\mathbb{R}^n \setminus \{0\}.$$ 
	
	\section{Proof of Main Result}\label{sec3}
	
	To prove Theorem \ref{th1.1}, we divide the argument into multiple propositions.  Consider \(\omega_1 \subset \subset \mathcal{O}  \) as a non-empty open set.  Then, there exists another non-empty open set \(\omega_2 \subset \subset \mathcal{O}  \), such that \[ \overline{\omega_1} \cap \overline{\omega_2} = \emptyset. \]
	\begin{proposition}\label{p1}
		Under the assumptions of Theorem \ref{th1.1}, for any \( f \in C_0^\infty(\omega_1) \) and for all \( k = 0,1,2, \dots \), we have  
		\[
		\int_0^\infty \phi(s)\, s^k \, ds = 0,
		\]  
		where  
		\begin{equation}\label{phi}
			\phi(s) = \frac{\big(e^{-\frac{1}{s}\mathcal{L}_{g_1} } - e^{-\frac{1}{s}\mathcal{L}_{g_2} } \big) f (x)}{s^{\frac{1}{2}}}, \quad x \in \omega_2.
		\end{equation}
	\end{proposition}
	
	\begin{proof} The proof follows the steps outlined in \cite[Proof of Theorem 1.1]{FGKU25}.
		
		Since \( g_1 = g_2 = g \) on \( \omega_1 \), and \( C_0^\infty(\omega_1) \) is an invariant subspace under \( \Delta_g \), we obtain, for any \( k = 0,1,2, \dots \),  
		\[
		\Delta_{g_1}^{2} f = \Delta_{g_2}^{2} f = \Delta_g^{2} f\,\, \implies \mathcal{L}_{g_1} ^k f=\mathcal{L}_{g_2} ^k f=\mathcal{L}_{g} ^k f
		\quad \text{on } \omega_1.\]
		From our assumption, this leads to  
		\[ \mathcal{L}_{g_1} ^{-\frac{1}{2}}\mathcal{L}_{g} ^kf|_\mathcal{O}  = \mathcal{L}_{g_2} ^{-\frac{1}{2}}\mathcal{L}_{g} ^kf|_\mathcal{O} , \quad \forall f \in C_0^\infty(\omega_1), \quad k = 0,1,2, \dots. \notag
		\]
		Thus from \eqref{L-alpha} for $\alpha=\frac{1}{2}$ there and for $x\in \mathcal{O} $,  we have 
		\begin{equation}
			\int_0^\infty t^{-\frac{1}{2}}\, \left((e^{-t\mathcal{L}_{g_1} } - e^{-t\mathcal{L}_{g_2} })\,\,\mathcal{L}_{g} ^k f\right) (x)\,dt = 0,  \label{eq:heat_difference}
		\end{equation}
		or, by using the equation $\partial_t^k (e^{-t\mathcal{L}_{g_i}}v)=e^{-t\mathcal{L}_{g_i}}(\mathcal{L}_{g_i}^k v)$ in $(M_i, g_i)$ for any $v\in C^\infty_0(M_i)$ and $x\in \mathcal{O}$, 
		\begin{equation}  
			\int_0^\infty t^{-\frac{1}{2}}\,\left(\partial_t^k (e^{-t\mathcal{L}_{g_1} } - e^{-t\mathcal{L}_{g_2} }) f \right)(x) \,dt = 0. \label{3.1}
		\end{equation}
		
		Now, for \( l = 0,1,\dots,k-1 \), and for any \( x \in \omega_2 \) and \( f \in C_0^\infty(\omega_1) \), we express the heat kernel difference as  
		\begin{align}
			\partial_t^l ((e^{-t\mathcal{L}_{g_1} } - e^{-t\mathcal{L}_{g_2} }) f)(x) &= ((e^{-t\mathcal{L}_{g_1} } - e^{-t\mathcal{L}_{g_2} })\,\,\mathcal{L}_{g} ^k f)(x) \notag \\
			&= \int_{\omega_1} (\mathcal{K}_{\mathcal{L}_{g_1} }(t,x,y) - \mathcal{K}_{\mathcal{L}_{g_2} }(t,x,y))  \mathcal{L}_{g} ^k f(y) \, dV_g(y). \label{3.2}
		\end{align}
		
		Thus, for \( t > 0 \), \( x \in \omega_2 \), and \( k = 0,1,2,\dots \), we obtain the bound  
		\begin{equation}
			\left |\partial_t^l ((e^{-t\mathcal{L}_{g_1} } - e^{-t\mathcal{L}_{g_2} }) f)(x) \right | \leq \| \mathcal{K}_{\mathcal{L}_{g_1} } -\mathcal{K}_{\mathcal{L}_{g_2} } \|_{L^\infty(\omega_1 \times \omega_2)}\,\, \|  \mathcal{L}_{g} ^k f \|_{L^1(\omega_1)}. \label{3.3}
		\end{equation}
		
		\medskip
		For \( t \in (0,1) \), using \eqref{2.2} and \eqref{3.3}, we obtain for \( l = 0,1,2,\dots \)  
		\begin{equation}\begin{aligned}
				\left |\partial_t^l ((e^{-t\mathcal{L}_{g_1} } - e^{-t\mathcal{L}_{g_2} }) f)(x) \right | &\leq C e^{-m^2t} e^{ -c_1/t^{\frac{1}{3}}}  \|  \mathcal{L}_{g} ^k f \|_{L^2(\omega_1)}, \label{3.4}\\
				&\to 0 \quad \text{as } t \to 0,
		\end{aligned}\end{equation}
		where, \( c_1 >0 \) depends on \( d_g(\overline{\omega_1}, \overline{\omega_2}) \).

		According to \cite[Theorem 1]{Varo85}, 
		\begin{align}
			\|e^{-t\mathcal{L}_{g_i}}  \mathcal{L}_g ^{2l} f\| \leq \frac{C}{t^\frac{n}{2}} \,\,  \| \mathcal{L}_{g} ^k f \|_{L^1(\omega_1)} \quad \text{for}\,\, t>0,
		\end{align}
		which implies that for   \( t \in [1,\infty) \) and \( l=0,1,2,\dots,k-1 \), applying to \eqref{3.2}, yields  
		\begin{equation}\begin{aligned}
				\left |\partial_t^l ((e^{-t\mathcal{L}_{g_1} } - e^{-t\mathcal{L}_{g_2} }) f)(x) \right | &\leq \frac{C}{t^\frac{n}{2}} \,\,  \|  \mathcal{L}_{g} ^k f \|_{L^1(\omega_1)}, \label{3.5}\\
				&\to 0 \quad \text{as } t \to \infty.
		\end{aligned}\end{equation}
		Now,  using integration by parts in \eqref{3.1} along with the decay estimates \eqref{3.4} and \eqref{3.5}, we obtain  
		\begin{equation}
			\int_0^\infty  (e^{-t\mathcal{L}_{g_1} } - e^{-t\mathcal{L}_{g_2} }) f \,\,t^{-(k+\frac{1}{2})} \,dt = 0 \quad \text{for } k=1,2,\dots
		\end{equation}
		Rewriting the above for \(k = 0,1,2,\dots \), we have \[ \int_0^\infty (e^{-t\mathcal{L}_{g_1} } - e^{-t\mathcal{L}_{g_2} }) f \,\,t^{-(k+1+\frac{1}{2})} \,dt = 0. \]
		
		Finally, by changing the variable \(s = \frac{1}{t} \), we get at \[ \int_0^\infty \phi(s)\, s^k \, ds = 0, \]
		where, $\phi(s) = \frac{\big(e^{-\frac{1}{s}\mathcal{L}_{g_1} } - e^{-\frac{1}{s}\mathcal{L}_{g_2} } \big) f}{s^{\frac{1}{2}}}$.  This concludes the proof.
	\end{proof}
	
	\begin{proposition}\label{p2}
		Here $\phi(s) \in L^2(0,\infty)$, and the function  defined by 
		\begin{equation}\label{f-phi}
			f(z):=\int_0^\infty \phi(s)\, e^{2\pi i z\,s}
		\end{equation} is holomorphic on
		$\mathbb{C}^+:=\{ x+iy : y> 0\}$.
	\end{proposition}
	\begin{proof}
		Using \eqref{3.4} and \eqref{3.5}, for $l=0$ there, we get
		$$|\phi(s)|\leq C s^\frac{n}{2} \quad \text{for}\,\, s\in (0,1] \quad \text{and}\,\,\, |\phi(s)|\leq C e^{-c_1 s^{\frac{1}{3}}} \text{for}\,\, s\in (1, \infty). $$
		
		Observed that,  
		$$\int_0^1 |\phi(s)|^2 ds \leq C^2 \int_0^1 s^n ds < \infty,$$
		and 
		\begin{align*}
			\int_1^\infty |\phi(s)|^2 ds &\leq C^2 \int_1^\infty e^{-2c_1 s^{1/3}} ds\\
			&=3C^2 \int_{1}^{\infty} e^{-2c_1 s} \,s^2 ds <\infty.
		\end{align*}
		Since both integrals are finite, it follows that
		\begin{equation}\label{phiL2} \phi\in L^2(0,\infty).\end{equation} 
		
		\medskip
		Let $f$ is given as in \eqref{f-phi}. We estimate \( f(z) \) as
		\begin{align*}
			\left| f(z) \right| &\leq \int_0^\infty |\phi(s)| e^{-2\pi\,y\,s} ds, \quad z=x+\mathrm{i}y, \mbox{ with }y>0, \\
			&\leq C \|\phi\|_{L^2(0,\infty)}.
		\end{align*}
		Thus, it is well defined pointwise on \( \mathbb{C}^+ \). 
		
		Let \( z \in \mathbb{C}^+ \) and let $\{z_n\}_{n\in\mathbb{N}}$ be a sequence in $\mathbb{C}^{+}$ such that \( z_n \to z \). Then, one finds
		\begin{align*}
			\left|  f(z_n) - f(z) \right|^2 & \leq \left(  \int_0^\infty  |\phi(s)| |e^{2\pi i z_n s}- e^{2\pi i z s}| ds \right)^2 \\
			&\leq \|\phi\|_{L^2}^2 \int_0^\infty  | \phi(s)| |e^{2\pi i z_n s}- e^{2\pi i z s} |^2 ds.
		\end{align*}
		For both \( \operatorname{Im}(z) > \delta > 0 \), and  \( \operatorname{Im}(z_n) > \delta \), gives
		
		\[
		|e^{2\pi i z_n s}- e^{2\pi i z s} | \leq 4 e^{-2\delta s}.
		\]
		Therefore,  using the Lebesgue dominated convergence theorem, we conclude that \( f\) is continuous at \( z \). 
		
		\medskip
		Next, let \( T \) be any closed triangle in \( \mathbb{C}^+ \). Then, one sees 
		\begin{align*}
			\int_T f(z) dz &= \int_T \int_0^\infty \phi(s) e^{2\pi i z s} ds dz \\
			&= \int_0^\infty \phi(s) ds \int_T e^{2\pi i z s} dz = 0.
		\end{align*}
		Thanks to Morera's theorem \cite{Conw73}, we conclude that \( f \) is holomorphic in \( \mathbb{C}^+ \). This completes the proof.
	\end{proof}
	
	\begin{proposition}\label{p3}
		\( f(z) \) defined in \eqref{f-phi} is identically zero over $\mathbb{C}^{+}$. 
	\end{proposition}
	
	\begin{proof}
		Let 
		\begin{equation}
			I= \int_0^\infty \sum_{k=0}^\infty \frac{(2\pi i \,z\,s)^k}{k!} \phi(s) \, ds.\label{e1}
		\end{equation}
		Then, we estimate \( |I| \) as follows:
		
		\begin{align*}
			|I| &\leq \int_0^\infty \sum_{k=0}^\infty \frac{(2\pi  \,|z|\,s)^k}{k!} |\phi(s)| \, ds \\
			&\quad=  \int_0^\infty e^{2\pi |z|s} |\phi(s)| \, ds\\
			&\quad= \int_0^1 e^{2\pi |z|s} |\phi(s)| \, ds + \int_1^\infty e^{2\pi |z|s} |\phi(s)| \, ds\\
			&\quad\leq  C \int_0^1 \underbrace{e^{2\pi |z|s}  s^{\frac{n-1}{2}}}_{h_z(s)} ds\,\, +\int_1^\infty \underbrace{e^{2\pi |z|s} e^{-\frac{m^2}{s}} s^{\frac{n-2}{4}} e^{-c_1 s^{\frac{1}{3}}}}_{g_z(s)} \, ds\\
			&\qquad=C\int_0^1 h_z(s) \, ds + \int_1^\infty g_z(s) \, ds.
		\end{align*}
		Here, we see that $ \left| \int_0^1 h_z(s) \, ds \right| < \infty$.
		For the following integral, define \(c_2(z) = -c_1 + 2\pi |z| \), where $c_1$ is the constant from \eqref{3.4}, and let
		\[ D= \{z\in \mathbb{C}^+: c_2(z) < 0\}, \quad \text{i.e.,} \quad |z|< \frac{c_1}{2\pi}. \]
		Then, given \(z\in D \), we estimate the second integral (using \eqref{3.5}):
		\begin{align*}
			\left| \int_1^\infty g_z(s) \, ds \right| 
			&\leq \int_1^\infty  e^{2\pi |z|s} s^{\frac{n-2}{4}} e^{-c_1 s^{\frac{1}{3}}} \, ds\\
			&\leq \int_1^\infty e^{(-c_1 + 2\pi |z|)s^{\frac{1}{3}}} s^{\frac{n-2}{4}} \, ds\\
			&= \int_1^\infty e^{c_2(z)s^{\frac{1}{3}}} s^{\frac{n-2}{4}} \, ds.
		\end{align*}
		Since \(c_2(z) < 0 \) for \(z \in D \), the integral converges, meaning that $ I < \infty \quad \forall z \in D.$ Using Fubini's theorem, we can interchange summation and integration on \eqref{e1} to obtain:
		$$\,\,\, \int_0^\infty \sum_{k=0}^\infty \frac{(2\pi i \,z\,s)^k}{k!} \phi(s) \, ds =\sum_{k=0}^\infty 
		\int_0^\infty \frac{(2\pi i \,z\,s)^k}{k!} \phi(s), \, ds,\quad \forall z \in D.$$
		Let $\Omega \subset\subset D$ be a non-empty, simply connected open set. For each $a \in \Omega$, the Cauchy integral formula yields
		\begin{align*}
			f(a) &= \frac{1}{2\pi i} \int_{\partial \Omega} \frac{f(z)}{z - a} \, dz \\
			&= \frac{1}{2\pi i} \int_{\partial \Omega} \frac{1}{z - a} \, dz \int_0^\infty \phi(s)\, e^{2\pi i z s} \, ds\\
			&= \frac{1}{2\pi i} \int_{\partial \Omega} \frac{1}{z - a} \, dz \int_0^\infty \sum_{k=0}^\infty \frac{(2\pi i z s)^k}{k!}  \phi(s) \, ds\\
			&= \frac{1}{2\pi i} \sum_{k=0}^\infty \int_{\partial \Omega} \frac{(2\pi i z)^k}{(z - a) k!} \, dz \int_0^\infty  \phi(s) s^k \, ds=0.
		\end{align*}
		As $f(z)$ is holomorphic in \(\mathbb{C}^+ \), the identity theorem implies that \[ f(z) = 0 \quad \forall z \in \mathbb{C}^+. \]
		This concludes the proof.
	\end{proof}
	\begin{proposition}\label{p4}
		$\phi$, defined in \eqref{phi}, is $0$ over $(0,\infty)$. \end{proposition} 
	\begin{proof}
		Let's recall \eqref{f-phi}
		\begin{equation}\label{f-1} f(z) = \int_0^\infty \phi(s)\, e^{2\pi i z\,s},\end{equation}
		where $\phi\in L^2(0, \infty)$  thanks to \eqref{phiL2}.   According to Proposition \ref{p2}, $f$ is a holomorphic function over $\mathbb{C}^{+}$. 
		
		\medskip
		At this point, we define the so called Hardy space $H^2(\mathbb{C}^+)$, that is the space of holomorphic functions in $\mathbb{C}^{+}$, such that \begin{equation}\label{hardy}\|f\|^2_{H^2(\mathbb{C}^+)}:= \sup_{y>0} \int_{-\infty}^{\infty} |f(x+iy)|^2dx\,<\infty.
		\end{equation}
		By the classical Paley–Wiener theorem \cite{PW87}, given any $\phi\in L^2(0,\infty)$, defining $f$ as in \eqref{f-1}, we have that $f\in H^2(\mathbb{C}^{+})$, and \begin{equation}\label{norm} \|f\|_{H^2(\mathbb{C}^+)}=\|\phi\|_{L^2(0,\infty)}. 
		\end{equation}
		As we have already demonstrated that \(f(z) = 0 \) for all \(z \in \mathbb{C}^+ \) (cf. Proposition \ref{p3}), it follows from \eqref{norm} that \(\phi(s) \equiv 0 \) for all \(s > 0 \).  This concludes the proof. 
	\end{proof}
	
	\noindent
	This leads to our final finding:
	\begin{proposition}\label{p5}
		$$\mathcal{K}_{\Delta_{g_1}}(t,x,y)=\mathcal{K}_{\Delta_{g_2}}(t,x,y)\quad \text{$\forall t>0$  and $x,y \in O$},$$
		where $\mathcal{K}_{\Delta_{g_i}}(t,x,y)$ is the heat kernel of the heat semigroup $e^{t\Delta_{g_i}}$ in $(0,\infty)\times (M_i, g_i)$ for $i=1,2$ respectively.  
	\end{proposition}
	
	\begin{proof}
		Using Propositions \ref{p1} - \ref{p4}, we have that, for all $f\in C^\infty_c(\omega_1)$,
		\[
		e^{-t\mathcal{L}_{g_1} }f(x) = e^{-t\mathcal{L}_{g_2} }f(x), \quad\forall \,\, t>0, \mbox{ and }x\in\omega_2.
		\]
		This implies, for all $f\in C^\infty_c(\omega_1)$,
		\begin{equation}\label{tran1}
			e^{-t\Delta^2_{g_1}}f(x)=e^{-t\Delta^2_{g_2}}f(x), \quad\forall \,\, t>0, \mbox{ and }x\in\omega_2. 
		\end{equation}
		
		\medskip
		Next, we use the transmutation formula of Kannai (cf. \cite{Kann77}), which says, for all $v \in C^\infty(M_i)$,
		\begin{equation}
			e^{-t\Delta^2_{g_i}} v = \frac{1}{4\sqrt{\pi}t^{\frac{1}{3}}} \int_0^\infty e^{-\frac{\tau}{4t}} \frac{\sin(\sqrt{\tau} \Delta_{g_i})}{\Delta_{g_i}} v\, d\tau,\quad \forall t>0,
		\end{equation}
		where $i = 1,2$.
		
		\medskip
		Let $f \in C_0^\infty(\omega_1)$. Then from \eqref{tran1} and \eqref{tran1}, we have
		\begin{equation}
			\int_0^\infty e^{-\frac{\tau}{4t}} \left(\frac{\sin(\sqrt{\tau} \Delta_{g_1})}{\Delta_{g_1}} f\right)(x)\, d\tau = \int_0^\infty e^{-\frac{\tau}{4t}} \left(\frac{\sin(\sqrt{\tau} \Delta_{g_2})}{\Delta_{g_2}} f\right)(x)\, d\tau,
		\end{equation}
		for all $t > 0$, and $x \in \omega_2$. 
		
		\medskip
		Then by applying the inverse Laplace transform, we obtain
		\begin{equation}
			\left(\frac{\sin(\sigma \Delta_{g_1})}{\Delta_{g_1}} f\right)(x) = \left(\frac{\sin(\sigma \Delta_{g_2})}{\Delta_{g_2}} f\right)(x), \quad \forall \sigma > 0,\,\,\mbox{and } x \in \omega_2.
			\label{eq:sin_identity}
		\end{equation}
		
		Since on $\mathcal{O}\supset\omega_1, \omega_2$, we have $g_1 = g_2 = g$, it follows that
		
		\begin{equation}
			(\sin(\sigma \Delta_{g_1}) f)(x) = (\sin(\sigma \Delta_{g_2}) f)(x), \quad \forall \sigma > 0,\,\,\mbox{and } x \in \omega_2.
			\label{eq:sin_equality}
		\end{equation}
		Differentiating \eqref{eq:sin_identity} with respect to $\sigma$ yields
		\begin{equation}
			(\cos(\sigma \Delta_{g_1}) f)(x) = (\cos(\sigma \Delta_{g_2}) f)(x), \quad \forall \sigma > 0,\,\,\mbox{and } x \in \omega_2.
			\label{eq:cos_equality}
		\end{equation}
		By extending the sine function oddly and the cosine function evenly, we find that for all $\sigma \in \mathbb{R}$ and $x \in \omega_2$,
		\begin{align}
			(\sin(\sigma \Delta_{g_1}) f)(x) = (\sin(\sigma \Delta_{g_2}) f)(x),  \quad\mbox{and}\quad
			(\cos(\sigma \Delta_{g_1}) f)(x) = (\cos(\sigma \Delta_{g_2}) f)(x);
		\end{align}
		which gives
		\begin{equation}
			(e^{i\sigma \Delta_{g_1}} f)(x) = (e^{i\sigma \Delta_{g_2}} f)(x), \quad \forall \sigma \in \mathbb{R},\,\,\mbox{and } x \in \omega_2.
		\end{equation}
		
		\medskip
		Now for any $\phi \in \mathcal{S}(\mathbb{R})$, we can express
		\begin{equation*}
			(\phi(\Delta_{g_i}) f)(x) = \frac{1}{2\pi} \int_{\mathbb{R}} (e^{i\sigma \Delta_{g_i}} f)(x) \widehat{\phi}(\sigma)\, d\sigma.
		\end{equation*}
		Let us take $\psi \in C_0(\mathbb{R})$. then there exists a sequence ${\phi_j} \subset \mathcal{S}(\mathbb{R})$ such that $\phi_j \to \psi$ uniformly. In particular, by choosing $\psi(\sigma) = \frac{\sin{t\sqrt{|\sigma|}}}{\sqrt{|\sigma|}}$, we obtain
		
		\begin{equation}
			\left( \frac{\sin(t\sqrt{-\Delta_{g_1}})}{\sqrt{-\Delta_{g_1}}} \right) f(x) = \left( \frac{\sin(t\sqrt{-\Delta_{g_2}})}{\sqrt{-\Delta_{g_2}}} \right) f(x),\quad\forall t>0,\,\,\mbox{and }x\in\omega_2.
			\label{eq:wave_identity}
		\end{equation}
		
		\medskip
		Next, we seek Kannai's transmutation formula (see \cite{Kann77}) again, to write for every $v \in C^\infty(M_i)$
		\begin{equation}\label{tran2}
			e^{t\Delta_{g_i}} v = \frac{1}{4\sqrt{\pi}t^{\frac{1}{3}}} \int_0^\infty e^{-\frac{\tau}{4t}} \frac{\sin(\sqrt{\tau} \sqrt{-\Delta_{g_i}})}{\sqrt{-\Delta_{g_i}}} v\, d\tau, \quad\forall t>0,
		\end{equation}
		where $i = 1,2$.

		\medskip
		Thus using \eqref{eq:wave_identity} and \eqref{tran2}, we can deduce that for all $f\in C^\infty_c(\omega_1)$,
		\begin{equation*}
			e^{t\Delta_{g_1}} f(x)= e^{t\Delta_{g_2}} f(x), \quad\forall t>0,\,\,\mbox{and }x\in\omega_2.
		\end{equation*}
		Since the function \( (e^{t\Delta_{g_1}} - e^{t\Delta_{g_2}}) f \) restricted to \( (0,\infty) \times \overline{\omega_2} \) is smooth,
		\[
		\mbox{i.e.}\,\,    (e^{t\Delta_{g_1}} - e^{t\Delta_{g_2}}) f \in C^\infty((0,\infty) \times \overline{\omega_2}).
		\]
		And it satisfies the following heat equation
		\[
		(\partial_t - \Delta_g) (e^{t\Delta_{g_1}} - e^{t\Delta_{g_2}}) f = 0 \quad \text{in }(0,\infty)\times \mathcal{O},
		\]
		thanks to $g_1=g_2=g$ in $\mathcal{O}$.
		
		\medskip    
		Therefore, by the unique continuation property of the heat equation (cf. \cite{Lin90}), we conclude that, for all $f\in C^\infty_c(\omega_1)$,
		\[(e^{t\Delta_{g_1}} - e^{t\Delta_{g_2}}) f = 0 \quad \text{in }(0,\infty)\times \mathcal{O}.\]
		Since \( \omega_1\subset\mathcal{O} \) is arbitrary, it follows that,
		\[    e^{t\Delta_{g_1}} f = e^{t\Delta_{g_2}} f \quad \text{in }(0,\infty)\times \mathcal{O},
		\]
		for all $f\in C^\infty_c(\mathcal{O})$.
		
		\medskip
		This finally gives the equality of heat kernels:
		\[ \mathcal{K}_{\Delta_{g_1}}(t,x,y) = \mathcal{K}_{\Delta_{g_2}}(t,x,y), \quad \text{$\forall t>0,$  and $x,y \in O$}.\]
		This completes the proof.
	\end{proof}
	
	\noindent
	Finally, we use the following result from \cite[Theorem 1.5]{FGKU25} to finish the proof of our Theorem \ref{th1.1}.
	\begin{theorem}[\cite{FGKU25}]
		Let $(N_1, g_1)$ and $(N_2, g_2)$ be smooth connected
		complete Riemannian manifolds of dimension $n\geq 2$ without boundary. Let $\mathcal{O}_j \subset N_j$, $j = 1, 2,$ be open nonempty sets. Assume that
		$\mathcal{O}_1 =\mathcal{O}_2 := \mathcal{O}$. Assume furthermore that
		$$\mathcal{K}_{\Delta_{g_1}}(t, x, y) = \mathcal{K}_{\Delta_{g_2}}(t, x, y),\quad\forall t>0,\,\,\mbox{and }x,y\in\mathcal{O}.$$
		Then there exists a diffeomorphism $\varphi : N_1 \mapsto N_2$ such that $\varphi^\ast g_2 =g_1$ on $N_1$.
	\end{theorem} 
	\medskip
	This concludes the discussion on the proof of Theorem \ref{th1.1}.
	\qed

	\bigskip
	
	\subsection*{Acknowledgment:}

	I am grateful to Prof. Tuhin Ghosh (Harish-Chandra Research Institute, Prayagraj, India)  for his guidance and advice,  all of which helped me to accomplish this project. I'd also want to thank the Department of Atomic Energy, Government of India, for supporting this research through its Fellowship.

	{\small 
		\bibliographystyle{acm}
		\bibliography{Master_bibfile}}
\end{document}